%% file: super_duper_new_version.tex
\documentclass[12pt,a4paper]{amsart}

\usepackage[utf8]{inputenc}
\usepackage[OT1]{fontenc}
\usepackage{amsmath}
\usepackage{amsthm}
\usepackage{amssymb}
\usepackage{hyperref}
\usepackage{cleveref}
\usepackage{todonotes} 
\usepackage[left=2cm,right=2cm,top=2cm,bottom=2cm]{geometry}

\newtheorem{theorem}{Theorem}
\newtheorem*{theorem*}{Theorem}
\newtheorem{proposition}{Proposition}
\newtheorem{lemma}{Lemma}
\newtheorem{corollary}{Corollary}

\theoremstyle{definition}
\newtheorem{definition}{Definition}

\theoremstyle{remark}

\input{ex_shared}

\begin{document}

\title[Left-invariant sub-Riemannian Engel structures]{Left-invariant 
sub-Riemannian Engel structures: abnormal 
geodesics and integrability}

%\shorttitle{Left-invariant sub-Riemannian Engel structures}
\author{Ivan Beschastnyi}
\address{International School for Advanced Studies, via Bonomea 
 265, Trieste 34136, Italy}
\email{i.beschastnyi@gmail.com}

\author{Alexandr Medvedev}
\address{International School for Advanced Studies, via Bonomea 
 265, Trieste 34136, Italy}
\email{amedvedev@sissa.it}

\begin{abstract}
We provide the first known family of examples of integrable homogeneous sub-Riemannian structures admitting strictly abnormal geodesics. These examples were obtained through the analysis of the equivalence problem for sub-Riemannian Engel structures. We formulate a criterion of strict abnormality in terms of structure functions of a canonical frame on a sub-Riemannian Engel manifold as well as estimates on conjugate times.
\end{abstract}

\maketitle

\section{Introduction}
\label{sec:intro}

It is well known that in sub-Riemannian geometry every geodesic must be a projection of a curve in the cotangent bundle called an \emph{extremal} and that there are two mutually non-exclusive types of such curves: normal and abnormal extremals. By a \emph{geodesic} we mean a curve, whose short arcs are length minimizing. Until the relatively recent example by R.~Montgomery~\cite{montgomery_singular} it was believed that sub-Riemannian minimizers could be only normal. Projections of normal extremals are smooth, always locally minimizing and in many ways behave similarly to Riemannian geodesics. 

In 1994 Montgomery~\cite{montgomery_singular} provided the first example of an integrable sub-Rieman\-nian structure admitting strictly abnormal (i.e.~that are not projections of normal extremals) geodesics. Since then it became clear that abnormal extremals play a very important role in sub-Riemannian geometry~\cite{bonnard, 
hakavuori_nonminimal, agrachev_abnormal, sussmann}. For example Sussmann and Liu~\cite{SussmannLiu} showed that the projections of ``regular'' abnormal extremals for 2-distributions are locally minimizing. Their results suggest that abnormal geodesics are a typical phenomenon for 2-distributions. This fact is closely related to the rigidity phenomena for singular curves~\cite{BryantHsu}.

%In many cases this fact is related with the rigidity of singular curves~\cite{BryantHsu}, however not always~\cite{sussmann}. 

Compared to normal geodesics we know very little about abnormal geodesics. For example, it is not known what kind of singularities the sub-Riemannian wave-front and sphere can have in a neighborhood of an abnormal geodesic~\cite{agrachev_open}. Some results for non-strictly abnormal geodesics (i.e.~for those that are normal as well) were obtained in~\cite{agrachev_abnormal,sachkov,andryuha}. Most of them rely heavily on the fact that the considered models had an integrable Hamiltonian system for the normal extremals. In~\cite{bonnard} structures with both strictly and non-strictly abnormal geodesics were considered, but in the strict case the models did not have an integrable normal geodesic flow, and so only some limited results could be obtained.

The motivation for this paper was to find sub-Riemannian Engel structures that have strictly abnormal geodesics and an integrable Hamiltonian system for normal extremals. An Engel manifold is a 4-dimensional manifold together with a rank-2 distribution of growth (2,3,4). The global properties of Engel distributions have attracted a lot of interest lately (see~\cite{kaz1,kaz2} and references therein).
In \Cref{sec:1} using a canonical frame we show that sub-Riemannian Engel structures are locally defined by 6 structure functions $T_i, 1\le i \le 6$ (see \Cref{thm:1} for details). Then we provide a local classification of
left-invariant Engel structures on 4-dimensional Lie groups. 
This problem was previously considered in the works of 
Almeida~\cite{almeida1,almeida2}, but the classification there is 
incomplete.

Engel manifolds are foliated by abnormal geodesics~\cite{sussmann}. In \Cref{sec:2} we characterize Engel structures that admit strictly abnormal geodesics. That allows us to prove our main result: Theorem \ref{the:main}. We show that homogeneous sub-Riemannian Engel structures defined there are super-integrable and admit strictly abnormal geodesics. It is worth mentioning that Gol\'e and Karidi~\cite{karidi} have already shown that even Carnot groups could have strictly abnormal geodesics. However integrability was not addressed in their paper. Another closely related result is~\cite{montgomery_integrable} where the authors provided an example of a left-invariant sub-Riemannian structure on a Carnot group whose normal geodesic flow is not integrable.

In \Cref{sec:3} we discuss local optimality of an abnormal geodesic on a sub-Riemannian Engel manifold in its $C^0$-neighbourhood. The investigation is governed by the study of the Jacobi equations and corresponding conjugate points. Absence of conjugate points is sufficient for minimality if the abnormal geodesic is strict. In \Cref{thm:3} we show how the minimality of an 
abnormal geodesic is related to the behaviour of a function
\be
\Delta = T_6 + \frac12 \dot{T}_2 - \frac14 (T_2)^2
\ee  
along it. 

Function $\Delta$ is a curvature invariant similar to the curvature invariants of normal sub-Riemannian geodesics introduced in~\cite{zelenko1} by Zelenko and Li. Later this result was used by them in~\cite{zelenko2} and by Barilari and Rizzi in~\cite{comparison} to prove comparison theorems like \Cref{thm:3}.

In the left-invariant we compute case explicitly all conjugate times for abnormal geodesics. This times are equal to
\be
t_{conj} = \frac{\pi k}{\sqrt{\Delta}}, \qquad \forall k \in \Z_+.
\ee
This also gives a family of explicit examples of sub-Riemannian structures with abnormal geodesics that lose optimality at finite moments of time. Such examples were previously known (see~\cite{free_carnot}), but were mostly limited to non-strictly abnormal geodesics.
% ##abnormal geodesics should be strict or not-strict, not structures##
% ##Better to coment out## The left-invariant structures that we consider in this paper can have finite cut time and be either strict or not-strict.

\textbf{Acknowledgements.}
The authors would like to thank prof.~Andrei Agrachev and prof.~Yuri 
Sachkov for many useful discussions and suggestions. We also would like to thank the anonymous referees for the helpful remarks and valuable suggestions that helped to greatly improve the final text.

\section{Equivalence of Engel sub-Riemannian structures}\label{sec:1}
A four-di\-men\-sional manifold $M$ with a two-dimensional distribution 
$\cD$ is called an Engel manifold, if $\cD$ satisfies the following 
non-integrability conditions
\begin{align*}
&\rank ([\cD,\cD] )= 3,  \\
& \rank([\cD,[\cD,\cD]])  = 4, 
\end{align*}
where $[\cD, \cD]$ consists of those tangent vectors that can be 
obtained by taking commutators of local sections of $\cD$. 
Any two Engel structures are locally equivalent~\cite{engel,cartan}. A positive-definite metric $g$ on $\cD$ turns $M$ into a metric 
space with the distance
$$
d(q_0,q_1) = \inf_{\gamma} \left\{\int_0^1 \sqrt{g(\dot\gamma,\dot\gamma)}dt\right\},
$$
where the infimum is taken over all absolutely continuous curves $\gamma(t)$ such 
that $\gamma(0) = q_0$, $\gamma(1) = q_1$ and $\dot\gamma(t) \in 
\cD_{\gamma(t)} $ for almost every $t$. Locally minimizing curves are 
called \emph{geodesics}. 

It is well known that a sub-Riemannian Engel structure can be endowed with a canonical global frame. We shall now recall the construction for the reader's convenience.

Let $\mathcal{E}$ and $\mathcal{V}$ be distributions on an arbitrary manifold. We denote by $[\mathcal{E},\mathcal{V}]$ a distribution which 
is generated by brackets of germs of sections of $\mathcal{E}$ and 
$\mathcal{V}$. For an arbitrary distribution $\cD$ we use the notation 
$\cD^1=\cD$ and $\cD^i=[\cD,\cD^{i-1}].$
\begin{definition}
The \emph{Levi form} $\cL$ of a distribution $\cE$ in the point $p\in M$ is the  bi-linear and skew-symmetric map:
\[ \cL_p: \cE_p\times\cE_p \to T_p M/\cE_p, \]
defined by
\[  \cL_p(v,w)=[X_v,X_w]_p \mod{\cE} \]
where $X_v$ and $X_w$ are smooth vector fields defined in the neighborhood of $p$ which belong to $\cE$ and satisfy conditions $(X_v)_p=v$ and $(X_w)_p=w$.
\end{definition} 
It is straightforward  to check that the definition of the Levi form 
does not depend on the choice of $X_v$ and $X_w$.

\begin{lemma} Let $\cD$ be an Engel distribution and $\cL$ be the Levi 
form on $\cD^2$. Then the kernel $\cK$ of $\cL$ is one-dimensional and is 
contained in $\cD$.
\end{lemma}
\begin{proof}
	First of all  $\cL\neq0$ since $TM=[\cD,\cD^2]\subseteq [\cD^2,\cD^2]
	$. Therefore, the kernel is $1$-dimensional. Assume that 
	$\cK\not\subset
	\cD.$ 
	Then $\cD^2=\cK\op\cD$ and $[\cD,\cD^2]=[\cD,\cD\op\cK]=[\cD,\cD]=\cD^2$ 
	which contradicts the definition of an Engel structure.
\end{proof}
\begin{rem}\label{rem1}
	The kernel of $\cL$ in fact defines a characteristic line field of the 
	distribution $\cD$. Its integral lines are abnormal geodesics of any Engel sub-Riemannian structures defined on $(M,\cD)$. It is known that in the Engel case sufficiently short arcs of those curves are minimizers independently of the sub-Riemannian metric on the distribution~\cite{sussmann}.
\end{rem}

Let $\cK$ be the kernel of the Levi form $\cL$. With every 4-dimensional Engel structure we can associate a canonical, up to an action of $\Z_2\times\Z_2$, frame. Namely, let $X_2$ be one of the two unit vectors in  
$\cK$. Let $X_1$ be an orthogonal complement to $X_2$. Then, the vectors $X_3$ and $X_4$ are defined as follows:
\[ X_3 = [X_1,X_2], \qquad X_4 = [X_1,X_3]. \]

The frame $\{ X_1,X_2,X_3,X_4 \}$ is unique up to the action of a group 
$\Z_2\times\Z_2$ which is generated by the following $2$ elements:
\begin{align*}
& \{ X_1,X_2,X_3,X_4 \}\to \{ -X_1,X_2,-X_3,X_4 \}, \\
& \{ X_1,X_2,X_3,X_4 \}\to \{ X_1,-X_2,-X_3,-X_4 \}.
\end{align*}
Note that we can omit the $\Z_2\times\Z_2$ ambiguity by fixing the frame orientation as well as the orientation of the sub-frame $X_1,X_2$.

Every sub-Riemannian Engel structure induces a canonical filtration of 
the tangent bundle of $M$:
\be
0=F^0\subset F^{-1}\subset F^{-2}\subset F^{-3}\subset F^{-4}=TM,
\ee
where $F^{-1}=\R X_1$, $F^{-2}=\cD$ and $F^{-3}=\cD^2$. The associated graded Lie algebra $\gr F=\bigoplus_i F^{-i}/F^{-i+1}$ 
is isomorphic to the standard nilpotent Engel Lie algebra
\[ [e_1,e_2]=e_3, \,\, [e_1,e_3]=e_4. \] 
The projection $ \Gamma(TM)\to\gr F$ sends $X_i$ to $e_i$.  

Consider now the structure functions of the canonical frame
\be
[X_i,X_j]=X_{i+j}+C_{ij}^k X_k, \,\, i<j.
\ee
The grading $\deg(e_i)=-i$ of the standard Engel Lie algebra induces corresponding grading of the frame $\deg(X_i)=-i$ and of the structure function $\deg(C_{ij}^k)=i+j-k.$ Since the frame is compatible with the filtration $F$, all non-zero $C_{ij}^k$ have a positive degree, i.e. $C_{ij}^k=0$ if $i+j-k \le 0$.

It is well known that the structure constants of a frame together with all covariant derivatives (i.e. derivatives by the vector fields forming the frame) form a set of invariants which is sufficient to solve an equivalence problem\cite{sternberg1983}.
Not all $C_{ij}^k$ are independent. Due to Jacobi identity, we can express them using 6 basic invariants $T_1, \dots , T_6$, which we describe in the table below.
\begin{table}[h]
\caption{Basic invariants of Sub-Riemannian Engel structures}
\begin{center}
\begin{tabu}{ cc }
\toprule
Degree & Invariant \\
\midrule
1 & $T_1 = C_{14}^4$ \\
\midrule
2 & $T_2 = C_{23}^3$ \\
  & $T_3 = C_{14}^3$ \\
\midrule
3 & $T_4 = C_{23}^2$ \\
  & $T_5 = C_{14}^2$ \\
\midrule
4 & $T_6 = C_{23}^1$ \\
\bottomrule
\end{tabu}
\end{center}
\end{table}
\\
The next theorem proves this fact. 
\begin{theorem}\label{thm:1}
For every oriented sub-Riemannian Engel structure 
$(M,\cD,g)$ with fixed orientation on $\cD$ there exists a canonical frame $\{X_1,X_2,X_3, X_4 \}$ given 
by conditions
\be\label{e:triv}
	[X_1,X_2]= X_3,\,\, [X_1,X_3] = X_4,\,\,
	[X_2,X_3]\in \spn{X_1,X_2,X_3},
\ee
such that orientations of $\{X_1,X_2,X_3, X_4 \}$ and $\{X_1,X_2\}$ are compatible with orientations of $M$ and $\cD$ respectively.

Apart from \eqref{e:triv} the structure equations of the canonical frame are:
\be\label{eq:struct}
\begin{aligned}
%	&[X_1,X_2]= X_3,\,\, [X_1,X_3] = X_4,\\
  &[X_1,X_4]=C^1_{14} X_1+T_5 X_2+T_3 X_3+T_1 X_4 \\
	&[X_2,X_3]=T_6 X_1 + T_4 X_2 + T_2 X_3 \\
	&[X_2,X_4]=X_1(T_6) X_1 + X_1(T_4) X_2 + (T_4 + X_1(T_2) ) X_3 
	+ T_2 X_4 \\
	&[X_3,X_4]=C_{34}^1 X_1 + C_{34}^2 X_2 + C^3_{34} X_3 + 
	(T_4+2 X_1(T_2)-X_2(T_1)) X_4,
\end{aligned}
\ee
where 
\begin{align*}
C^1_{14}&=\frac12 \left(T_1 T_4 + T_1 X_1(T_2) - 3 X_1(T_4) + X_2(T_3) + X_3(T_1) - 
X_1^2(T_2)\right),
\\
C^3_{34}&=-\frac12\left(T_1 T_4 + T_1 X_1(T_2) - X_1(T_4) + X_2(T_3) - X_3(T_1) - 
X_1^2(T_2)\right),
\\
C^2_{34}&=T_2 T_5 - T_3 T_4 - T_1 X_1(T_4) - X_2 (T_5) + X_1^2(T_4) ,
\\
C^1_{34}&= T_2 C^1_{14} - T_6 T_3 - T_1 X_1(T_6) - 
X_2(C^1_{14})+ X_1^2(T_6).
\end{align*}
In particular, the structure constants depend only on $T_i$, $1\le i \le 6$ and their derivatives along $X_j$.
\end{theorem}
\begin{proof}
The normalization conditions imply some restrictions on the form of structure functions:  equality $C^4_{12} = C^2_{12}=C^1_{12}=0$ follows from $[X_1,X_2]=X_3$, equality $C^3_{13}=C^2_{13}=C^1_{13}=0$ follows 
from 
$[X_1,X_3]=X_4$ and $C^4_{23}=0$ follows from the fact that $X_2$ 
generates the kernel of the Levi form on $\cD^2$. 

The Jacobi identity yields relations among the $C^k_{ij}$, which we compute in order according to their degree. In degree 1 there remains only one non-zero structure function $C_{14}^4$ and we denote it by $T_1$. 
In degree 2 there is only one non-trivial relation which follows from the Jacobi identity. Denoting by $[Y,Z]^i$ the projection of $[Y,Z]$ on $X_i$ we have:
\[ 0= [X_1,[X_2,X_3]]^4 + [X_2,[X_3,X_1]]^4 = [X_1,C_{23}^3X_3]^4-[X_2,X_4]^4=C^3_{23}-C^4_{24}. \]
Therefore we are allowed to use the notation $C^3_{23}=C^4_{24}=T_2$ and $C^3_{14}=T_3$. In degree 3 there are 2 relations. The first one is
\begin{multline*} 0= [X_1,[X_2,X_3]]^3 + [X_2,[X_3,X_1]]^3 =
\\ [X_1,C^2_{23}X_2+C_{23}^3X_3]^3-C_{24}^3=C^2_{23}+X_1(C_{23}^3)-C^3_{24}
\end{multline*}
while the second one is
\begin{multline*} 0= [[X_1,X_2],X_4]^4 + [[X_2,X_4],X_1]^4 +  [[X_4,X_1],X_2]^4  =
\\ 
C_{34}^4 + [C^3_{24}X_3 + C_{24}^4X_4,X_1]^4 - [C_{14}^4X_4,X_2]^4= C_{34}^4 - C^3_{24}-X_1(C^3_{23})+X_2(C_{14}^4)
\end{multline*}
Therefore if we define $T_4=C_{23}^2$ and $T_5=C_{14}^2$ then $C^3_{24}=T_4+X_1(T_2)$ and $ C_{34}^4 = T_2+2X_1(C^3_{23})-X_2(T_1)$. 

We omit the computations for higher degrees since they are more involved, but straightforward.
\end{proof}

Let us consider the classification problem for the
left-invariant Engel sub-Rieman\-nian structures on Lie groups. The structure functions are constant in this case. The following general 
form of the structure equations for the canonical left-invariant frame is a 
direct consequence of \cref{thm:1}.
\begin{proposition} \label{prop2}
Let $\{X_1,X_2,X_3, X_4 \}$	be a canonical left-invariant frame for a left-invariant Engel sub-Riemannian structure. Then the structure equations of the 
frame are:
\be
\begin{aligned}
	&[X_1,X_2]= X_3,\,\, [X_1,X_3] = X_4,\\
  &[X_1,X_4]=\frac12 A X_1+T_5 X_2+T_3 X_3+T_1 X_4, \\
	&[X_2,X_3]=T_6 X_1 + T_4 X_2 + T_2 X_3, \\
	&[X_2,X_4]=T_4 X_3 + T_2 X_4, \\
	&[X_3,X_4]=C X_1+B X_2-\frac12 A X_3+T_4 X_4,
\end{aligned}
\ee
where $A=T_1 T_4$, $B=T_2 T_5-T_3 T_4$, $C=
\frac12 T_1 T_2 T_4-T_3 T_6$.
\end{proposition}

Substituting the structure constants from 
\cref{prop2} into the Jacobi formula we obtain a system of 
restrictions on $T_i$:
\begin{equation}\label{eq4}
\begin{aligned}
0 &= T_1 T_6  + 2 T_2 T_4 ,
\\
0 &= T_1^2 T_4  + 4T_2 T_5 ,
\\
0 &= T_1 T_3 T_4 - T_1 T_2 T_5 + 2 T_4 T_5 ,
\\
0 &= T_1 T_4^2 - T_1^2 T_2 T_4  + 2 T_1 T_3 T_6 + 2T_5 T_6 
,
\\
0 &= T_1 T_4^2 + 4 T_2^2 T_5 - 4 T_2 T_3 T_4 + 2T_5 T_6,
\\ 
0 &= T_1 T_2^2 T_4 + T_1 T_4 T_6 - 2 T_2 T_3 T_6. 
\end{aligned}
\end{equation} 
Solving the system above we get the classification of the left-invariant 
sub-Riemannian Engel structures.
\begin{theorem}
Any left-invariant sub-Riemannian Engel structure is 
uniquely locally defined by the structure constants $T_i$ and belongs to at least one family from Table \ref{tbl_a}. We list in Table \ref{tbl_a} restrictions on $T_i$ that define a family as well as corresponding non-trivial structure equations.
\begin{table}[h]
\caption{Classification of left-invariant Sub-Riemannian Engel structures}
\label{tbl_a}
\begin{tabu}{ l  *2{>{$} l<{$}} }
\toprule
\mbox{\#} & \mbox{Restrictions} &\mbox{Structure Equations}\\ 
& &\mbox{Excluding }[X_1,X_2]=X_3,\,\,[X_1,X_3]=X_4\\
\midrule
I. & T_2 = T_4 = T_6 = 0 &  [X_1,X_4] = T_5 X_2 + T_3 X_3 + T_1 X_4
\\
\midrule
II. & T_4 = T_6 = T_5 = 0 & 	[X_1,X_4] = T_3 X_3 + T_1 X_4,
\\
& &[X_2,X_3] = T_2 X_3,
\\
& & [X_2,X_4]= T_2 X_4 	
\\ 
\midrule	
III. & T_1 = T_2 = T_5 = 0 &
[X_1,X_4] = T_3 X_3,
\\
& & [X_2,X_3] =T_6 X_1 + T_4 X_2, 
\\
& &   [X_2,X_4] = T_4 X_3,
\\
& & [X_3,X_4] = -T_6 T_3 X_1 -T_4 T_3 X_2 + T_4 X_4
\\
\midrule
IV. & T_1 = T_3 = 0, & 	[X_2,X_3] =T_6 X_1 + T_2 X_3,  \\
&  T_4 = T_5 = 0 & [X_2,X_4] =  T_2 X_4. \\
\midrule
V. & T_1\neq 0, &
 [X_1, X_4] = T_1X_4-\frac{T_1^3+8T_5}{4T_1}X_3 +T_5X_2-\frac{2T_2T_5}{T_1}X_1,\\
& T_4 = \frac12 \frac{T_2(T_1^2+4T_3)}{T_1},&
 [X_2, X_3] = T_2X_3-\frac{4T_2T_5}{T_1^2}X_2+\frac{8T_2^2T_5}{T_1^3}X_1, \\
& T_5 = -\frac18 T_1^3-\frac12 	 T_1 T_3,&
 [X_2, X_4] = T_2X_4-\frac{4T_2T_5}{T_1^2}X_3,\\
& T_6 = -\frac{T_2^2(T_1^2+4T_3)}{T_1^2} 
& [X_3, X_4] = \frac{2T_2T_5}{T_1}\left(X_3 -\frac{2}{T_1} X_4-\frac{4T_5}{T_1^2} X_2+\frac{8T_2T_5}{T_1^3} X_1\right)\\
\bottomrule
 \end{tabu}
 \end{table}
\end{theorem}
%\begin{rem}
%As one can see families from the theorem above have non-trivial 
%intersections.
%\end{rem}

\section{Geodesic flow on Engel manifolds and its integrability}
\label{sec:2}

The problem of finding length minimizers between $q_0,q_T\in M$  
is equivalent to the optimal control problem
\begin{equation}
\label{eq:control_sys}
\dot{q} = u_1 X_1(q) + u_2 X_2(q), \qquad u_1,u_2 \in L^\infty\left([0,T]\right),\,\, q(t)\in M,
\end{equation}
\begin{equation}
q(0) = q_0, \qquad q(T) = q_T,
\end{equation}
\begin{equation}
\label{eq:length}
l(q) = \int_0^T \sqrt{u_1^2 +u_2^2}dt \rightarrow \min,
\end{equation}
where $u_i$ are the controls. 

\begin{definition}
An \emph{admissible curve} is a Lipschitz curve that satisfies \eqref{eq:control_sys} at almost every point. 
\end{definition}

It is well known that the minimum exists and that after a reparameterization we can assume that minimal curves have constant speed $|\dot q|^2 = u_1^2 +u_2^2 = const$. Then a direct consequence of the Cauchy-Schwarz inequality implies that an admissible curve with constant speed is a length minimizer if and only if it minimizes the energy functional
\begin{equation}
J(q) = \int_0^T \frac{u_1^2 +u_2^2}{2}dt
\end{equation}
with $T$ fixed~\cite{ABB}.

\begin{definition}
A \emph{geodesic} is an admissible curve parametrized by constant speed whose sufficiently small arcs are length minimizers.
\end{definition}
In order to describe sub-Riemannian geodesics we use the Pontryagin maximum principle (PMP) which is equivalent to the usual Lagrange multiplier rule in constrained optimization. To state it we need some definitions.

\begin{definition}
We say that the pair control-trajectory $(\tilde u(t),\tilde q(t))$ is an \emph{optimal pair}, if $\tilde{q}(t)$ is a length minimizer and satisfies \eqref{eq:control_sys} with control function $u=\tilde{u}(t)$.
\end{definition}

Consider the cotangent bundle $\pi\colon T^*M \to M$ and the coordinate functions 
\[h_i = \langle \lambda, X_i \rangle,\,\,\lambda \in T^* M.\]
%on the fibres of $T^*M$. 
\begin{definition}
The \emph{Hamiltonian of the maximum principle} is a family of smooth functions, affine on fibres, parameterized by controls $(u_1,u_2)\in \R^2$ and a real number $\nu \leq 0$, given by
$$
H_u (\lambda,\nu) = \langle \lambda, u_1 X_1 + u_2 X_2 \rangle + \frac{\nu}{2}(u_1^2 + u_2^2) = u_1h_2 + u_2 h_2 + \frac{\nu}{2}(u_1^2 + u_2^2).
$$
\end{definition}
\begin{theorem}[PMP,~\cite{as}]
\label{thm:pmp} 
If a pair $(\tilde u(t),\tilde q(t))$ is optimal in a minimization problem \eqref{eq:control_sys}-\eqref{eq:length}, then there exists a 
Lipschitzian curve $\lambda(t) \in T^*_{\tilde{q}(t)} M$ and a number 
$\nu \leq 0$, s.t. the following conditions are satisfied
\begin{enumerate}
\item\label{pmp:1} $(\lambda(t),\nu)\neq 0$;
\item $\dot\lambda(t) = \vec{H}_{\tilde{u}(t)}(\lambda(t))$;
\item\label{pmp:3} $H_{\tilde{u}(t)} = \max_{u\in \R^2} 
H_u(\lambda(t),\nu) $.
\end{enumerate}
\end{theorem}

\begin{definition}
The curve $\lambda(t) \in T^*_{\tilde{q}(t)} M$ from the formulation of PMP is called an \emph{extremal}. 
\end{definition}

\begin{definition}
If an extremal $\lambda(t)$ satisfies the PMP with $\nu = 0$ it is called 
\emph{abnormal} or \emph{singular}, otherwise we say that it is \emph{normal}. We say that a projection $\tilde{q}(t)$ of $\lambda(t)$ is normal (resp. abnormal) if the corresponding extremal $\lambda(t)$ is normal (resp. abnormal). A curve $\tilde{q}(t)$ is said to be strictly abnormal (resp. strictly normal) if it satisfies the PMP with some abnormal (resp. normal) $\lambda(t)$ and is not a projection of some normal (resp. abnormal) extremal at the same time.
\end{definition}
A projection of a normal extremal is always a geodesic. In the case of sub-Riemannian Engel structures a projection of any abnormal extremal is a locally minimizing curve i.e. it is a geodesic~\cite{ABB,sussmann}, although it is not true in general for other sub-Riemannian structures.

We now look for sub-Riemannian structures of Engel type admitting strictly abnormal geodesics.
\begin{theorem}\label{prop:2}
Abnormal geodesics of an Engel sub-Riemannian structure are integral curves of $X_2$. An abnormal geodesic is strict if and only if $T_4 \neq 0$ along the geodesic.
\end{theorem}
\begin{proof}
Consider a geodesic $q(t)$ and let $\lambda(t)=(q(t),h(t))$ be its extremal. The Hamiltonian system of the PMP for a sub-Riemannian Engel structure is given by
\begin{align*}
&\dot{q} = u_1 X_1(q) + u_2 X_2(q),\\
&\dot{h}_i = \{H_u,h_i\}.
\end{align*}
where the Lie-Poisson bracket of vertical coordinate functions $h_i$ 
depends only on the structure functions of $X_i$:
\[\{h_i,h_j\} = \langle\lambda, [X_i,X_j] \rangle= C_{ij}^k(q) h_k. \]
Using the structure equations \cref{eq:struct} and the Leibniz rule we obtain
\begin{align}
\dot{q} &= u_1 X_1(q) + u_2 X_2(q)\nonumber,\\
\dot{h}_1 &= -u_2 h_3\nonumber,\\
\dot{h}_2 &= u_1 h_3, \label{eq:h_sys} \\
\dot{h}_3 &= u_1 h_4 + u_2\left(T_6 h_1 + T_4 h_2 + T_2 
h_3\right),\nonumber \\
\dot{h}_4 &= u_1 \left( C^1_{14} h_1 + T_5h_2 + T_3 h_3 + T_1 
h_4 
\right) \nonumber
\\
       & + u_2\left(X_1(T_6) h_1 + X_1(T_4) h_2 + (T_4  + X_1(T_2))h_3 
       + 
       T_2 
h_4\right).\nonumber
\end{align}
To find the optimal controls $u=(u_1,u_2)$ we use condition 
\eqref{pmp:3} from \cref{thm:pmp}.
When $\nu = 0$, the 
Hamiltonian is of the form
\[
H_u = u_1 h_1 + u_2 h_2.
\]

The only possibility for the maximum to be attained is when $h_1 \equiv h_2 
\equiv 0$. This implies $\dot h_1 = -u_2 h_3=0$ and $\dot h_2= u_1 h_3=0$. Since we are interested in 
curves with non-zero constant speed $\dot q= u_1^2 + u_2^2$ we obtain that $h_3\equiv 0$. The forth equation of \eqref{eq:h_sys} implies that either 
$u_1 \equiv 0$ or $h_4 \equiv 0$. But the non-triviality condition \ref{pmp:1} of the PMP yields $h_4\neq 0$ if $\nu=0$. Therefore, $u_1 \equiv 0$ and projections of abnormal extremals are integral curves of $X_2$. Along these curves the last equation reduces to  
\[\dot{h}_4 = u_2 T_2 h_4\]
whose solutions are sign-definite for non-zero initial data. Therefore 
the non-triviality condition is satisfied for all times and  $(q(t),h(t))$ is an abnormal extremal. Moreover $q(t)$ is always a length minimizer the in Engel case~\cite{ABB}, i.e. it is always an abnormal geodesic.

Let us consider the case $\nu \neq 0$. Without loss of 
generality we can normalize $(\lambda,\nu)$ in such a way that $\nu = 
-1$.  Then the maximum is achieved when
$$
\frac{\p H_u}{\p u_i} = h_i - u_i = 0 \iff u_i = h_i, \qquad i=1,2. 
$$
Substituting the obtained controls in \cref{eq:h_sys} we get
\begin{align}
\dot{q} &=  h_1 X_1(q) + h_2 X_2(q)\nonumber,\\
\dot{h}_1 &=  -h_2 h_3\nonumber,\\
\dot{h}_2 &=  h_1 h_3, \label{eq:h_sys_norm} \\
\dot{h}_3 &=  h_1 h_4 + h_2\left(T_6 h_1 + T_4 h_2 + T_2 
h_3\right),\nonumber \\
\dot{h}_4 &=  h_1 \left( C^1_{14} h_1 + T_5h_2 + T_3 h_3 + T_1 
h_4 
\right) \nonumber
\\
&+ h_2\left( X_1(T_6)h_1 + X_1(T_4)h_2 + (T_4  + X_1(T_2))h_3 + 
T_2 h_4 \right).\nonumber
\end{align}
which is a Hamiltonian system with Hamiltonian
\[
H_{\tilde{u}(t)} = H = \frac{h_1^2 + h_2^2}{2}.
\]

Assume that an abnormal geodesic $(q(t),h(t))$ satisfies 
\eqref{eq:h_sys_norm}. Since it is an integral curve of $X_2$ we must 
have $h_1\equiv 0$. Moreover, the Hamiltonian $H$ is a first integral of the 
system. Therefore $2H = h_1^2 + h_2^2 = const \neq 0$ and so $h_2 
= const \neq 0$. Thus from \eqref{eq:h_sys_norm} it follows that $h_3 
\equiv 0$. But the forth equation 
gives us $T_4 h_2^2 = 0$, which can hold if and only if $T_4 = 0$ along the curve. 
All  these conditions reduce the system to the equation 
\be\label{eq:red}
\dot h_4 = h_2 T_2 h_4,
\ee
which always has a solution.
 
On the other hand assume that along an abnormal extremal $T_4=0$. By substituting 
$h_1\equiv 0$, $h_2\equiv 1$, $h_3\equiv 0$ into \eqref{eq:h_sys_norm} we reduce the 
system to \eqref{eq:red}. This equation always has a sign-definite solution which 
guarantees that the abnormal extremal is normal as well.
\end{proof}

One can check from the classification in section~\ref{sec:2}, that among the type III left-invariant Engel structures, there are indeed those that have $T_4 \neq 0$. The following result says that the normal geodesic flow on all these algebras is integrable. 
 \begin{theorem}
\label{the:main}
\label{prop:integra}
Consider a left-invariant sub-Riemannian Engel structure of type III which is defined over a Lie group with a Lie algebra
\begin{align}
&[X_1,X_2]=X_3, & &[X_1,X_3]=X_4, 
\\
&[X_1,X_4] = T_3 X_3, & &[X_2,X_3] =T_6 X_1 + T_4 X_2, 
\\
&[X_2,X_4] = T_4 X_3, & &[X_3,X_4] = -T_6 T_3 X_1 -T_4 T_3 X_2 + T_4 X_4,
\end{align}
where vector fields $X_1,$ $X_2$ form an orthonormal sub-Riemannian frame. The normal Hamiltonian flow of this structure is super-integrable meaning that it has four independent commuting first integrals including the Hamiltonian $H$ and one more independent first integral that commutes 
with $H$. If $T_4 \neq 0$ then the abnormal geodesics of the structure are strict.
\end{theorem} 

Before we prove the theorem, let us investigate the structure of corresponding Lie algebras. First, one can notice that any type III Lie algebra is a central extension of a 3-dimensional Lie algebra. The center element is
\[
X'_4 = X_4 + T_4 X_1 -T_3 X_2.
\]

The underlying Lie algebra is semi-simple if and only if $D=(T_4)^2+ T_3 T_6\neq 
0$. If $D<0$ and $T_3<0$ (equivalently $T_6<0$) then it is 
$\gso(3,\R)$ or $\gsl(2,\R)$ otherwise. 

Consider now the case $D=0$. If $T_4=T_6=0$ then we have a trivial extension either of the Lie algebra of Euclidean motions of the plane ($T_3>0$) or the Lie algebra 
of Poincare motions of the plane ($T_3<0$).
Otherwise $T_4\neq 0$ and we obtain a non-trivial 
extension of a solvable Lie algebra of dimension 3 with 
2-dimensional derived algebra.
The whole family already appeared in the classification of Almeida in~\cite{almeida2} and among examples of Engel structures in~\cite{gershkovich}.

\begin{proof}[Proof of Theorem~\ref{prop:integra}]
Instead of the basis for the type III family from \Cref{tbl_a}, we use basis $\{X_1,X_2,X_3,X'_4\}$ in the proof. Then the only non-zero structure equations are
\begin{align}
[X_1,X_2] &= X_3,\\
[X_1,X_3] &= X'_4 - T_4 X_1 + T_3 X_2,\\
[X_2,X_3] &=  T_6 X_1 + T_4 X_2.
\end{align}
The Hamiltonian function $h'_4 = \langle \lambda, X'_4\rangle$ which corresponds to the center element $X_4'$ is a first integral. In the basis $X_1,X_2,X_3,X'_4$ the Hamiltonian system takes the form
\begin{align}
\dot{h}_1 &= -h_2 h_3, \nonumber\\
\dot{h}_2 &= h_1 h_3, \label{eq:ham_ver}\\
\dot{h}_3 &= h_1h'_4 - T_4(h_1^2 - h_2^2) + (T_3 + T_6)h_1h_2,\nonumber
\end{align}
where $h'_4 = const$. It is easy to see that \eqref{eq:ham_ver} has the following first integral 
$$
G = \frac{h_3^2}{2} - h'_4 h_2 + \frac{T_3+T_6}{4}(h_1^2 - h_2^2) + T_4 
h_1h_2.
$$

Let $I: g \mapsto g^{-1}$ be the inverse map of the Lie algebra and $X^R(g) = I_*X^L(g)$ be the right invariant fields 
constructed from the left-invariant ones. 
Let $h^R_i$ be the right-invariant Hamiltonian functions. We know that $h^R_i$ commute with any left-invariant 
Hamiltonian function and thus they commute with $H,G,h'_4$. Therefore the whole family $H,G,h'_4,h^R_1$ is commutative and $\{H, h^R_2\} = 0$.

We claim that $dH,dG,dh'_4,dh^R_1,dh^R_2$ are linearly independent 
almost everywhere. Actually it is enough to check that only in one, point for example, at the identity. Indeed it is known that any finite-dimensional Lie group   is analytic, i.e. it admits an analytic structure as a manifold with analytic multiplication. Then the right and left-invariant Hamiltonians and their differentials are going to be analytic as well. Since the linear dependence is an algebraic condition on the components of the corresponding vectors, we get that if the differential above are linearly independent at some point, then they must be independent almost everywhere.

Assume that left-invariant Hamiltonian functions and right invariant 
Hamiltonian functions are related by
\[
 h^R_i =  a^j_i(g) h_j,
\]
where $X^R_i(g) = a^j_i(g) X^L_j(g)$ and $a^j_i(\id) = -\delta^j_i$. In~\cite{sachkov} it was shown that  in the coordinates of the first 
kind $
\frac{\p a^k_i}{d x^j}(\id) = c^k_{ji}.
$
Using this identity we deduce
\[
\begin{psmallmatrix}
dH\\
dG\\
dh'_4\\
dh^R_1\\
dh^R_2
\end{psmallmatrix} = 
\begin{psmallmatrix}
h_1 & h_2 & 0 & 0 & 0 & 0 & 0 & 0 \\
\frac{T_3+T_6}{2}h_1 +T_4h_2 & -h'_4 -\frac{T_3+T_6}{2}h_2 +T_4h_1 & h_3 
& -h_2 & 0 &0 &0 &0 \\
0 & 0 & 0 & 1 & 0 & 0 &0 &0 \\
-1 & 0 & 0 & 0 & 0 & -h_3 & -h'_4 + T_4 h_1 - T_3 h_2 & 0 \\
0 & -1 & 0 & 0 & h_3 & 0 & - T_6 h_1 - T_4 h_2 & 0
\end{psmallmatrix}.
\]
The determinant of the first, third, fourth, fifth and sixth rows is equal to 
$h_1 h_3^3$. Therefore $H,G,h'_4,h^R_1,h^R_2$ are 
almost everywhere functionally independent first integrals.
\end{proof}

\begin{rem}
It is straightforward to verify that the normal Hamiltonian flow of the left-invariant Engel structure admits Casimir functions only for examples of type I and type III. As follows from the structure equations, type I Lie algebras do not admit strictly abnormal geodesics, but from the integrability point of view they are simpler and could be worth considering. For example, type I algebras with structure constants 
\[ 
T_1 = n + m - 1, \quad
T_3 = n + m - nm, \quad
T_6 = - nm,
\]
admit polynomial first integrals of order $n+1$ and $m+1$, with any $m > n \geq 0$, which are given by
\begin{align*}
F_1 & =  \left( \frac{h_3 + h_4 
-(h_2+h_3)n}{(1+m)(m-n)} \right) \left( \frac{ h_4 + 
mnh_2-(m+n)h_3}{(1+m)(1+n)} \right)^m , \\
F_2 &  =  \left( \frac{m(h_2 + h_3) - h_3 - 
h_4}{(1+n)(m-n)} \right) \left( \frac{ h_4 + 
mnh_2-(m+n)h_3}{(1+m)(1+n)} \right)^n.
\end{align*}
\end{rem}

\section{Local minimality of abnormal geodesics}
\label{sec:3}

In differential geometry local minimality is usually understood in the sense that sufficiently short arcs of a curve are minimal. That means that for every point $t_0$ on an admissible curve $\gamma$ there exists a sufficiently small interval $[t_1,t_2]$ containing $t_0$ such that $\gamma|_{[t_1,t_2]}$ is the shortest curve among all admissible curves connecting $\gamma(t_1)$ and $\gamma(t_2)$. Such a curve is called geodesic.

However, in calculus of variations the word local  in  ``local minimality'' often refers to topology on a space of admissible curves. Consider a curve $\gamma$ defined on $[0,T]$. We are interested whether the whole curve $\gamma$ is shorter then any other sufficiently close admissible curve connecting $\gamma(0)$ and $\gamma(T)$. The answer to this question depends heavily on the topology we choose. Sobolev space topology $W^{1,\infty}$ was studied in~\cite{sussmann} for Engel manifolds and~\cite{agrachev_abnormal} for the general case. Some results on the $C^1$ topology can be found in~\cite{BryantHsu} and local optimality conditions for rank 2 distributions in the $C^0$-topology can be found in Chapter 12 of~\cite{ABB}. We follow the last reference.

\begin{definition}
An admissible curve $\gamma$ connecting $\gamma(0) = q_0$ with $\gamma(T) = q_T$ is called a \textit{$C^0$-local minimizer} if there exists a $C^0$-neighbourhood $U$ of $\gamma$, s.t. any other admissible curve $\hat{\gamma}$ {\color{blue} from $U$ with $\hat{\gamma}(0) = q_0$ and $\hat{\gamma}(T) = q_T$   is not longer then $\gamma$.}
%## l(.) isn't used anywhere else## $\hat{\gamma}$ satisfies $l(\gamma)\leq l(\hat{\gamma})$.
\end{definition}

The Pontryagin maximum principle guarantees that short arcs of normal curves are length minimizers. The analysis of local minimality of abnormal curves is a subtle question in general. However, for Engel manifolds short pieces of abnormal curves are $C^0$-local minimizers. This was proven in~\cite{ABB} by first establishing that abnormal curves on an Engel manifold are $H^1$-local minimizers and then by showing that $H^1$-local minimality implies $C^0$-local minimality for continuously differentiable curves.

To determine whether or not the whole geodesic is a $C^0$-local minimizer we investigate the presence of conjugate points along it. We present here only definitions and the theory related to the Engel case. For the most general situation see~\cite{agrachev_maslov} and for some particular cases see~\cite{as,ABB}.

\begin{definition}[\cite{ABB}]
Let $\gamma(t) = e^{tX_2}(q_0)$ be a unit speed abnormal geodesic on an Engel manifold. The moment of time $t>0$ is called conjugate if 
$$
e^{tX_2}_* \cD_{q_0} =  \cD_{\gamma(t)}.
$$
\end{definition}

\begin{theorem}[\cite{ABB}]
If an abnormal geodesic of an Engel manifold does not contain conjugate points, then it is a $C^0$-local minimizer. Conversely, if a \textbf{strictly} abnormal geodesic is a $C^0$-local minimizer, then it does not contain conjugate points.
\end{theorem}

It is important to note that in general a presence of a conjugate point does not imply that the abnormal geodesic is not a $C^0$-minimizer. The minimizing property depends on the number of lifts this geodesic has. If it has a unique lift to the cotangent bundle, then indeed a presence of at least one conjugate point is sufficient for non-optimality. For the general case see~\cite[Theorem~20.3]{as}. 

The next theorem establishes necessary conditions for $C^0$-local minimality of abnormal geodesics.
\begin{theorem}\label{thm:3}
Let $\gamma(t) = e^{tX_2}(q_0)$ be an unit-speed abnormal geodesic on an Engel 
manifold and let
$$
\Delta_\gamma(t) = T_6(\gamma(t)) + \frac12 \dot T_2(\gamma(t)) - \frac14 T_2(\gamma(t))^2.
$$
If $\Delta_\gamma \leq 0$ on $[0,T]$, then $\gamma|_{[0,T]}$ is
$C^0$-local minimizing. If $\gamma(t)$ is strictly abnormal and $\Delta  \geq C > 0$, then $\gamma|_{[0,\tau]}$ is not a $C^0$-local minimizer for $\tau \geq \pi/\sqrt{C}$. 
\end{theorem}
\begin{proof}
Let us write down and analyse the corresponding Jacobi equation. 
Obviously $e^{tX_2}_* (X_2(\gamma(0)))= X_2(\gamma(t))\in \cD_{\gamma(t)}$. So we must 
consider the evolution of $A(t) = e^{tX_2}_* X_1$ along the abnormal 
curve $\gamma(t)$. A time $t_*>0$ is conjugate if and only if 
$A(t_*)(\gamma(t_*)) \in \cD_{\gamma(t_*)}$. Using the definition of  Lie derivative we see that
\begin{equation}\label{eq:jac}
\dot{A}(t) = [A(t),X_2].
\end{equation}

Let ${A}(t)=A_1(t)X_1+A_2(t)X_2+A_3(t)X_3+A_4(t)X_4.$ Using the structure constants \eqref{eq:struct} of the canonical frame and projecting equation \eqref{eq:jac} on $\{X_1,X_2,X_3,X_4\}$ we obtain
\be\label{eq:jacobi}
\begin{aligned}
\dot{A}_1 &= -T_6 A_3 - X_1(T_6) A_4,\\
\dot{A}_2 &= -T_4 A_3 - X_1(T_4) A_4,\\
\dot{A}_3 &= A_1 - T_2 A_3 - (T_4  + X_1(T_2)) A_4,\\
\dot{A}_4 &= -T_2 A_4,
\end{aligned}
\ee
where all $T_i$ as well as $X_1(T_i)$ are evaluated along the curve $\gamma$. The system \eqref{eq:jacobi} is linear with boundary conditions $A(0) = 
(1,0,0,0)$ and $A_3(t_*) = A_4 (t_*) = 0$ where $t_*$ is the supposed
conjugate time. Note that the first, the third and the fourth equations form a closed subsystem. Moreover from the last 
equation and the boundary conditions we obtain $A_4 \equiv 0$. This way we are left to study the non trivial solutions to the boundary value problem
\be\label{eq:bound}
\begin{aligned}
\dot{A}_1 &= -T_6 A_3,
\\
\dot{A}_3 &= A_1 - T_2 A_3,
\\
A_1(0)= 1,\;& A_3(0) =0, \; A_3(t_*) = 0.
\end{aligned}
\ee
Using the fact that the abnormal curve is smooth since it is an integral curve of a smooth vector field, we rewrite 
\eqref{eq:bound} as a single second order ODE:
\begin{equation*}
\begin{aligned}
&\ddot{A}_3 + T_2 \dot{A}_3 + (T_6 + \dot{T}_2)A_3 = 0, \\
A_3&(0)= 0 ,\; A_3(t_*) \; = 0, \; \dot{A}_3(0) = 1.
\end{aligned}
\end{equation*}

This allows us to use the results from the oscillation theory of 
second order ODEs. After the change of variables 
$$
y = A_3 \exp\left( \int_0^t  \frac{T_2(\tau)}{2} d\tau \right),
$$
we get an equivalent formulation of the boundary value problem
\begin{align}
&\ddot{y} +\left( T_6 + \frac{\dot{T}_2}{2} - \frac{T_2^2}{4} \right)y = 
0, \label{eq:jacobi_trans}\\
y&(0) = 0 ,\; y(t_*) \; = 0, \; \dot{y}(0) = 1.\label{e:bvp} 
\end{align}
Now the statement of the theorem is a direct consequence of the 
Sturm comparison theorem.
\end{proof}

In the case of left-invariant structures, i.e. when all 
$T_i$'s are constants, we get a sharp result.

\begin{corollary}
Let $\gamma(t)$ be an abnormal curve of a left-invariant Engel structure and let $\Delta = T_6  - \frac14 (T_2)^2$. If 
$\Delta   > 0$, then all the conjugate times are given by
$$
t_{conj} = \frac{\pi k}{\sqrt{\Delta}}, \qquad \forall k \in \Z_+
$$ 
and if, moreover, $\gamma(t)$ is strictly abnormal then the restriction $\gamma|_{[0,\tau]}$ is a $C^0$-local minimizer if and only if $\tau < \pi/\sqrt{\Delta}$. If $\Delta \leq 0$,  then the restriction $\gamma|_{[0,\tau]}$ is a $C^0$-local minimizer for any $ \tau \in (0,+\infty)$. 
\end{corollary}
\begin{proof}
In the left-invariant case $\Delta$ is a constant. Therefore we can solve the boundary value problem \eqref{eq:jacobi_trans}-\eqref{e:bvp} explicitly.
\end{proof}

\bibliographystyle{siam}
\bibliography{references}
\end{document}

%% file: ex_shared.tex
\usepackage{lipsum}
\usepackage{amsfonts}
\usepackage{graphicx}
\usepackage{epstopdf}
\usepackage{microtype}
\usepackage{booktabs}
\usepackage{tabu}
\usepackage{multirow}
\usepackage{mathtools}
\usepackage{hyphenat}
\usepackage{cite}
\usepackage{mydef}

\usepackage{amsopn}